\RequirePackage{fix-cm}
\documentclass[smallcondensed]{svjour3}     
\smartqed  
\usepackage{graphicx,subfigure,amsmath,amssymb,latexsym,amsfonts,amstext,multirow}
\usepackage[mathscr]{eucal}
\newcommand{\me}{\mathrm{e}}

\newcommand{\dif}{\mathrm{d}}
\numberwithin{equation}{section}

\begin{document}
\titlerunning{On the initial value problem for the Navier-Stokes equations}
\title{Well-posedness for the Navier-Stokes equations with datum in the Sobolev spaces}
\author{D. Q. Khai}
\institute{Dao Quang Khai \at Institute of Mathematics, Vietnam Academy of Science and Technology, 18 Hoang Quoc Viet, 10307  Cau Giay, Hanoi, Vietnam  \\
              \email{khaitoantin@gmail.com}
           }
\date{Received: date / Accepted: date}
\maketitle
\begin{abstract}
In this paper, we study local well-posedness for the Navier-Stokes \linebreak equations with arbitrary initial data in homogeneous Sobolev spaces $\dot{H}^s_p(\mathbb{R}^d)$ for $d \geq 2, p > \frac{d}{2},\ {\rm and}\ \frac{d}{p} - 1 \leq s  < \frac{d}{2p}$. The obtained result improves the known ones for $p > d$ and $s = 0$ (see \cite{M. Cannone 1995,M. Cannone. Y. Meyer 1995}). In the case of critical indexes $s=\frac{d}{p}-1$, we prove global well-posedness for Navier-Stokes equations when the norm of the initial value is small enough. This result is a generalization of the one in  \cite{M. Cannone 1997} in which $p = d$ and $s = 0$.
\end{abstract}

\keywords{Navier-Stokes equations, existence and uniqueness of local and global mild solutions, critical Sobolev and Besov spaces.}

\subclass{Primary 35Q30; Secondary 76D05, 76N10.}


\section{Introduction}
We consider the Navier-Stokes equations (NSE) in $d$ dimensions 
in special setting of a viscous, homogeneous, incompressible fluid 
which fills the entire space and is not submitted to external forces. 
Thus, the equations we consider are the system 
\begin{align} 
\left\{\begin{array}{ll} \partial _tu  = \Delta u - 
\nabla\cdot(u \otimes u) - \nabla p , & \\ 
\nabla\cdot u = 0, & \\
u(0, x) = u_0, 
\end{array}\right . \notag
\end{align}
which is a condensed writing for
\begin{align} 
\left\{\begin{array}{ll} 1 \leq k \leq d, \ \  \partial _tu_k  
= \Delta u_k - \sum_{l =1}^{d}\partial_l(u_lu_k) - \partial_kp , & \\ 
\sum_{l =1}^{d}\partial_lu_l = 0, & \\
1 \leq k \leq d, \ \ u_k(0, x) = u_{0k} .
\end{array}\right . \notag
\end{align}
The unknown quantities are the velocity $u(t, x)=(u_1(t, x),\dots,u_d(t, x))$ 
of the fluid element at time $t$ and position $x$ and the pressure $p(t, x)$.\\
A translation invariant Banach space of tempered distributions $\mathcal{E}$ 
is called a critical space for NSE if its norm is invariant under the action 
of the scaling $f(.) \longrightarrow \lambda f(\lambda .)$. One can take, 
for example, $\mathcal{E} = L^d(\mathbb{R}^d)$ or the smaller 
space $\mathcal{E} = \dot{H}^{\frac{d}{2} - 1}(\mathbb{R}^d)$. 
In fact, one has the “chain of critical spaces” given by the continuous embeddings
\begin{equation}\label{eq1}
\dot{H}^{\frac{d}{2} - 1}(\mathbb{R}^d) \hookrightarrow 
L^d(\mathbb{R}^d) \hookrightarrow \dot{B}^{\frac{d}{p}-1}_{p, \infty}
(\mathbb{R}^d)_{(p < \infty)} \hookrightarrow BMO^{-1}(\mathbb{R}^d) 
\hookrightarrow \dot{B}^{- 1}_{\infty, \infty}(\mathbb{R}^d).
\end{equation}
It is remarkable feature that NSE are well-posed in the sense of Hadarmard 
(existence, uniqueness and continuous dependence on the data) when the initial 
datum is divergence-free and belongs to the critical function spaces 
(except $\dot{B}^{- 1}_{\infty, \infty}$) listed in 
$\eqref{eq1}$ (see \cite{M. Cannone 1995} for 
$\dot{H}^{\frac{d}{2} - 1}(\mathbb{R}^d)$, $L^d(\mathbb{R}^d)$, 
and $\dot{B}^{\frac{d}{p}-1}_{p, \infty}(\mathbb{R}^d)$,  
see \cite{H. Koch 2001} for $BMO^{-1}(\mathbb{R}^d)$, 
and the recent ill-posedness result  \cite{J. Bourgain 2008} 
for $\dot{B}^{- 1}_{\infty, \infty}(\mathbb{R}^d)$).\\
In the 1960s, mild solutions were first  constructed by Kato and 
Fujita \cite{H. Fujita 1964,T. Kato 1962} that are continuous 
in time and take values in the Sobolev space \linebreak 
$H^s(\mathbb{R}^d), 
(s \geq \frac{d}{2} - 1)$, say $u \in C([0, T); H^s(\mathbb{R}^d))$. 
In 1992, a modern treatment for mild solutions in 
$H^s(\mathbb{R}^d), (s \geq \frac{d}{2} - 1)$ 
was given by Chemin \cite{J. M. Chemin 1992}. In 1995, using 
the simplified version of the bilinear operator, Cannone proved 
the existence for mild solutions in $\dot{H}^s(\mathbb{R}^d), 
(s \geq \frac{d}{2} - 1)$, see \cite{M. Cannone 1995}. 
Results on the existence of mild solutions with value in 
$L^p(\mathbb{R}^d), (p > d)$ were established in  the papers 
of  Fabes, Jones and Rivi\`{e}re \cite{E. Fabes 1972a} 
and of Giga \cite{Y. Giga 1986a}. Concerning the initial 
datum in the space $L^\infty$, the existence of a mild solution was 
obtained by Cannone and Meyer in \cite{M. Cannone 1995,M. Cannone. Y. Meyer 1995}. Moreover, in \cite{M. Cannone 1995,M. Cannone. Y. Meyer 1995}, they also obtained theorems 
on the existence of mild solutions with value in the Morrey-Campanato space 
$M^p_2(\mathbb{R}^d), (p > d)$ and the 
Sobolev space $H^s_p(\mathbb{R}^d), (p < d, \frac{1}{p} 
- \frac{s}{d} < \frac{1}{d})$. NSE in the  
Morrey-Campanato space were also treated by Kato  \cite{T. Kato 1992} and 
Taylor \cite{M.E. Taylor 1992}. In 1981, Weissler \cite{F.B. Weissler:1981} gave the first existence result 
of mild solutions in the half space $L^3(\mathbb{R}^3_+)$. 
Then Giga and Miyakawa \cite{Y. Giga 1985} generalized the
result to $L^3(\Omega)$, where $\Omega$ is an open bounded domain in 
$\mathbb{R}^3$. Finally, in 1984, Kato \cite{T. Kato 1984} obtained, 
by means of a purely analytical tool (involving only the 
 H$\ddot{\text{o}}$lder and 
Young inequalities and without using any estimate of fractional 
powers of the Stokes operator), an existence theorem in the whole 
space $L^3(\mathbb{R}^3)$. In \cite{M. Cannone 1995,M. Cannone 1997}, Cannone showed how to simplify Kato's proof. The idea is to take the advantage of the structure of the bilinear operator in its scalar form. 
In particular, the divergence $\nabla$ and heat $\me^{t\Delta}$ operators can 
be treated as a single convolution operator. Recently, the authors of this article have considered NSE in Sobolev spaces, Sobolev-Lorentz spaces, mixed-norm Sobolev-Lorentz spaces, and Sobolev-Fourier-Lorentz spaces, see \cite{N. M. Tri: Tri2014????}, \cite{N. M. Tri: Tri2015??}, \cite{N. M. Tri: Tri2014a}, and \cite{N. M. Tri: Tri2015?} respectively. In \cite{N. M. Tri: Tri2016??}, we prove some results on the existence and decay properties of high order derivatives in time and space variables for local and global solutions of the Cauchy problem for NSE in Bessel-potential spaces. In \cite{N. M. Tri: Tri2016?}, we prove some results on the existence and  space-time decay rates of global strong solutions of the Cauchy problem for NSE equations in weighed $L^\infty(\mathbb R^d,|x|^\beta{\rm dx})$ spaces. In  \cite{N. M. Tri: Tri2014???}, we considered the  initial value problem for the non stationary NSE on torus $\mathbb T^3=\mathbb R^3/\mathbb Z^3$ and showed that NSE are well-posed when  the initial datum belongs to Sobolev spaces $V_{\alpha}: = D(-\Delta)^{\alpha /2}$ with $\frac{1}{2} < \alpha <\frac{3}{2}$. In this paper, we construct mild solutions in the spaces $C([0, T);\dot{H}^s_p(\mathbb{R}^d))$ to the Cauchy problem for NSE when  the initial datum belongs to the Sobolev spaces $\dot{H}^s_p(\mathbb{R}^d)$, with $d \geq 2, p > \frac{d}{2},\ {\rm and}\ \frac{d}{p} 
- 1 \leq s  < \frac{d}{2p}$. We obtain the existence of mild solutions with arbitrary initial value when $T$ is small enough;  and existence 
of mild solutions for any $T < +\infty$ when the norm of the initial value in the Besov spaces $\dot{B}^{s- d(\frac{1}{p} - \frac{1}{\tilde q}), 
\infty}_{\tilde q}$, $(\tilde q > {\rm max}\{p, q\},\ {\rm where}\ \frac{1}{q} = \frac{1}{p} - \frac{s}{d})$ is small enough. 
In the case $p > d$ and $s = 0$, this result is stronger than that of Cannone and Meyer \cite{M. Cannone 1995,M. Cannone. Y. Meyer 1995} but under a weaker condition on the initial data. In the case of critical indexes $(p > \frac{d}{2}, s = \frac{d}{p} - 1)$, we can take $T = \infty$ when the norm of the initial value in the Besov spaces 
$\dot{B}^{\frac{d}{\tilde q} - 1, \infty}_{\tilde q}(\mathbb{R}^d), (\tilde q > {\rm max}\{d, p\})$ is small enough. This result when 
$s=0$ and $p = d$ is the theorem of \linebreak Cannone \cite{M. Cannone 1997}. The content of this paper is as follows: in Section 2, we state our main theorem after introducing some notations. In Section 3, we first establish some estimates concerning the heat semigroup with differential. We also recall some auxiliary lemmas and several estimates in the homogeneous Sobolev spaces and Besov spaces. Finally, in Section 4, we will give the proof of the main theorem.
\section{Statement of the results}
Now, for $T > 0$, we say that $u$ is a mild solution of NSE on $[0, T]$ 
corresponding to a divergence-free initial datum $u_0$ 
when $u$ solves the integral equation
$$
u = e^{t\Delta}u_0 - \int_{0}^{t} e^{(t-\tau) \Delta} \mathbb{P} 
\nabla\cdot\big(u(\tau,.)\otimes u(\tau,.)\big) \dif\tau.
$$
Above we have used the following notation: 
for a tensor $F = (F_{ij})$ 
we define the vector $\nabla\cdot F$ by 
$(\nabla\cdot F)_i = \sum_{j = 1}^d\partial_jF_{ij}$ 
and for two vectors $u$ and $v$, we define their tensor 
product $(u \otimes v)_{ij} = u_iv_j$. 
The operator $\mathbb{P}$ is the Helmholtz-Leray 
projection onto the divergence-free fields 
$$
(\mathbb{P}f)_j =  f_j + \sum_{1 \leq k \leq d} R_jR_kf_k, 
$$
where $R_j$ is the Riesz transforms defined as 
$$
R_j = \frac{\partial_j}{\sqrt{- \Delta}}\ \ {\rm i.e.} \ \  
\widehat{R_jg}(\xi) = \frac{i\xi_j}{|\xi|}\hat{g}(\xi).
$$
The heat kernel $e^{t\Delta}$ is defined as 
$$
e^{t\Delta}u(x) = ((4\pi t)^{-d/2}e^{-|.|^2/4t}*u)(x).
$$
For a space of functions defined on $\mathbb R^d$, 
say $E(\mathbb R^d)$, we will abbreviate it as $E$. We denote by $L^q: = L^q(\mathbb R^d)$ the usual 
Lebesgue space for $q \in [1, \infty]$ with the norm $\|.\|_{L^q}$, 
and we do not distinguish between the vector-valued and scalar-valued spaces of functions. Given a Banach space $E$ with norm $\|.\|_E$, we denote by $L^p([0,T],E), 1\leq p \leq +\infty$, set of functions $f(t)$ defined on $(0,T)$ with values in $E$ such that $\int_0^T\|f(t)\|^p_E{\rm d}t <+\infty$. Let $BC([0,T);E)$ denote the bounded continuous functions defined on $(0,T)$. For vector-valued $f=(f_1,...,f_M)$, we define $\|f\|_E = \big(\sum_{m=1}^{m=M}\|f_m\|^2_E\big)^{\frac{1}{2}}$. We define the Sobolev space by $\dot{H}^s_q: = \dot{\Lambda}^{-s}L^q$ 
equipped with the norm 
$\big\|f\big\|_{\dot{H}^s_q}: =  \|\dot{\Lambda}^sf\|_{L^q}$. 
Here $\dot{\Lambda}^s:= \mathcal F^{-1}|\xi|^s\mathcal F$, 
where $\mathcal F$ and $\mathcal F^{-1}$ are the Fourier transform and its inverse, respectively. $\dot{\Lambda} = \sqrt{-\Delta}$ is the homogeneous Calderon pseudo-differential operator. Throughout the paper, we sometimes use the notation $A \lesssim B$ as an equivalent to $A \leq CB$ with a uniform constant $C$. The notation $A \simeq B$ means that $A \lesssim B$ and $B \lesssim A$. Now we can state our results
\begin{theorem}\label{th1} Let $p$ and $s$ be such that
$$
p > \frac{d}{2}\ and\ \frac{d}{p} - 1 \leq s  < \frac{d}{2p}.
$$
Set
$$
\frac{1}{q} = \frac{1}{p} - \frac{s}{d}.
$$
{\rm (a)}  For all $\tilde q  > {\rm max}\{p, q\}$, 
there exists a positive constant $\delta_{q,\tilde q,d}$ 
such that for all $T > 0$ and for all 
 $u_0 \in \dot{H}^s_p(\mathbb{R}^d)\ with\ {\rm div}(u_0) = 0$ satisfying
\begin{equation}\label{eq2}
T^{\frac{1}{2}(1+s-\frac{d}{p})}\underset{0 < t < T}{{\rm sup}}
t^{\frac{d}{2}(\frac{1}{p}- \frac{s}{d} - \frac{1}{\tilde q})}
\big\|\me^{t\Delta}u_0\big\|_{L^{\tilde q}} \leq \delta_{q,\tilde q,d},
\end{equation}
NSE has a unique mild solution $u \in BC([0, T); \dot{H}^s_p)$. Moreover, we have
$$
t^{\frac{d}{2}(\frac{1}{q}-\frac{1}{r})}u(t) \in  BC([0,T);L^r), \ for\ all\ r > {\rm max}\{p, q\}.
$$
In particular, the condition \eqref{eq2} holds for arbitrary  $u_0 \in \dot{H}^s_p(\mathbb{R}^d)$ when  $T=T(u_0)$ is small enough.\\
{\rm (b)} If $s = \frac{d}{p} - 1$ then for all 
$\tilde q  > {\rm max}\{p,d\}$ there exists a constant 
$\sigma_{\tilde q, d}>0$ such that if  
$\big\|u_0\big\|_{\dot{B}^{\frac{d}{\tilde q} - 1, \infty}_{\tilde q}}
 \leq \sigma_{\tilde q,d}$ and $T = +\infty$ 
then the condition \eqref{eq2} holds.
\end{theorem}
In the case of critical indexes $(s = \frac{d}{p} - 1, p > \frac{d}{2})$, we obtain the following consequence.
\begin{proposition}\label{pro1} Let $ p > \frac{d}{2}$. 
Then for any $\tilde q >{\rm max}\{p, d\}$, there exists a  positive 
constant $\delta_{\tilde q,d}$ such that for all $T > 0$ and 
for all $u_0 \in \dot{H}^{\frac{d}{p} - 1}_p(\mathbb{R}^d)$ 
with \linebreak ${\rm div}(u_0) = 0$ satisfying
\begin{equation}\label{eq3}
\underset{0 < t < T}{{\rm sup}}t^{\frac{1}{2}(1 - \frac{d}{\tilde q})}
\big\|\me^{t\Delta}u_0\big\|_{L^{\tilde q}} \leq \delta_{{\tilde q},d},
\end{equation}
NSE has a unique mild solution $u \in BC([0, T); \dot{H}^{\frac{d}{p}-1}_p)$. Moreover, we have
$$
t^{\frac{d}{2}(\frac{1}{d}-\frac{1}{r})}u(t) \in  BC([0,T);L^r), \ for\ all\ r > {\rm max}\{p, d\}.
$$
Denoting $w = u - \me^{t\Delta}u_0$ then  
$w \in BC([0, T); \dot{H}^{\frac{d}{\tilde p}-1}_{\tilde p})$ for all $\tilde p >  \frac{1}{2}{\rm max}\{p, d\}$.\\
In particular, the condition \eqref{eq3} holds for arbitrary  
$u_0 \in \dot{H}^{\frac{d}{p} - 1}_p(\mathbb{R}^d)$ when \linebreak $T=T(u_0)$ is small enough, and there exists a positive constant 
$\sigma_{\tilde q,d}$ such that \linebreak 
if  
$$
\big\|u_0\big\|_{\dot{B}^{\frac{d}{\tilde q} - 1, \infty}_{\tilde q}}
 \leq \sigma_{\tilde q,d}\  and\ T = +\infty,
$$
 then the condition \eqref{eq3} holds.
\end{proposition}
\begin{remark} If $p = d$ then Proposition \ref{pro1} is the theorem of Canone \cite{M. Cannone 1997}.
\end{remark}
In the case of supercritical indexes $p > \frac{d}{2}\ \text{and}\ \frac{d}{p} - 1 < s  < \frac{d}{2p}$, we get the following consequence.
\begin{proposition}\label{pro2}
Let $p > \frac{d}{2}\ \text{and}\ \frac{d}{p} - 1 < s  < \frac{d}{2p}$. 
Then for any $\tilde q$ be such that $\tilde q >{\rm max} \{p, q\},
$
where
$$
\frac{1}{q} = \frac{1}{p} - \frac{s}{d},
$$
there exists a positive constant $\delta_{q,\tilde q,d}$ 
such that for all $T > 0$ and for all \linebreak
 $u_0 \in \dot{H}^s_p(\mathbb{R}^d)\ with\ {\rm div}(u_0) = 0$ satisfying
\begin{equation}\label{eq4}
T^{\frac{1}{2}(1+s-\frac{d}{p})}\big\|u_0\big\|_{\dot{B}^{s- (\frac{d}{p} 
- \frac{d}{\tilde q}), \infty}_{\tilde q}} \leq \delta_{q,\tilde q,d},
\end{equation}
NSE has a unique mild solution $u \in BC([0, T); \dot{H}^s_p)$. Moreover, we have
$$
t^{\frac{d}{2}(\frac{1}{q}-\frac{1}{r})}u(t) \in  BC([0,T);L^r), \ for\ all\ r > {\rm max}\{p, q\}.
$$
\end{proposition}
\begin{remark} Proposition \ref{pro2} is the theorem of Canone and Meyer \cite{M. Cannone 1995,M. Cannone. Y. Meyer 1995} if $s=0$, $p > d$, and the condition \eqref{eq4} is replaced by the condition
\begin{equation*}
T^{\frac{1}{2}(1-\frac{d}{p})}\big\|u_0\big\|_{L^p} \leq \delta_{p,d}.
\end{equation*}
Note that in the case $s=0$ and $p > d$, the condition \eqref{eq4} is weaker than the above condition because of the following elementary imbedding maps
$$
L^p(\mathbb{R}^d) \hookrightarrow  \dot{B}^{-(
\frac{d}{p}-\frac{d}{\tilde q}), \infty}_{\tilde q}(\mathbb{R}^d), (\tilde q > p \geq d),
$$
but these two spaces are different. Indeed,  we have 
$\big|x\big|^{-\frac{d}{p}} \notin L^p(\mathbb{R}^d)$. 
On the other hand by using Lemma \ref{lem3}, we can 
easily prove that
$\big|x\big|^{-\frac{d}{p}} \in \dot{B}^{
-(\frac{d}{p}-\frac{d}{\tilde q}), \infty}_{\tilde q}(\mathbb{R}^d)$ 
for all $\tilde q > p$.
\end{remark}

\section{Tools from harmonic analysis}
In this section we prepare some auxiliary lemmas.\\
The main property we use throughout this paper is that the operator $\dot\Lambda^s e^{t\Delta}\mathbb{P}\nabla$ is a matrix of convolution operators with bounded integrable kernels.
\begin{lemma}\label{lem1}
Let $s>-1$. Then the kernel function of  $\dot{\Lambda}^se^{t\Delta} \mathbb{P}\nabla$ is the function
$$
K_t(x)= t^{-\frac{d+1+s}{2}}K\big(\frac{x}{\sqrt t}\big),
$$
where the function $K$ is the kernel function of  $\dot{\Lambda}^se^{\Delta} \mathbb{P}\nabla$ which satisfies the following inequality
$$
|K(x)| \lesssim \frac{1}{1+|x|^{d+1+s}}.
$$
\end{lemma}
\begin{proof} See Proposition 11.1 \cite{P. G. Lemarie-Rieusset 2002}, p. 107. \end{proof} 
\begin{lemma}\label{lem2} {\rm(Sobolev inequalities)}.\\
If \ $s_1 > s_2, \ 1 < q_1, \ q_2 < \infty$, and $s_1 - \frac{d}{q_1} = s_2 - \frac{d}{q_2}$, then we have the following embedding mapping
$$
\dot{H}_{q_1}^{s_1} \hookrightarrow \dot{H}_{q_2}^{s_2}.
$$
\end{lemma}
In this paper we use the definition of the homogeneous 
Besov space  $\dot B^{s,p}_q$ in \linebreak \cite{G. Bourdaud 1993,G. Bourdaud 1988}. The following lemma will provide a different characterization 
of Besov spaces $\dot B^{s,p}_q$ in terms of the heat semigroup 
and will be one of the staple ingredients of the proof of Theorem \ref{th1}.
\begin{lemma}\label{lem3} \ \\
Let $1 \leq p, q \leq \infty$ and $ s < 0$. Then the two quantities
\begin{gather*}
\Big(\int_0^\infty(t^{-\frac{s}{2}}\big\|\me^{t\Delta}f\big\|_{L^q})^p
\frac{{\rm d}t}{t}\Big)^{\frac{1}{p}}\ and \ \big\|f\big\|_{\dot{B}_{q}^{s, p}} 
\ are \ equivalent.
\end{gather*}
\end{lemma}
\begin{proof} See Theorem 5.4 in (\cite{P. G. Lemarie-Rieusset 2002}, p. 45). \end{proof}
\begin{lemma}\label{lem4} Let $\theta< 1$ and $\gamma < 1$ then
$$
\int^t_0(t-\tau)^{-\gamma} \tau^{-\theta}{\rm d}\tau = C t^{1-\gamma - \theta},\  where\  
C = \int^1_0(1-\tau)^{-\gamma} \tau^{-\theta}{\rm d}\tau < \infty.
$$
\end{lemma}
The proof of this lemma is elementary and may be omitted. \qed
Let us recall following result on solutions of a quadratic
equation in Banach spaces (Theorem 22.4 in 
\cite{P. G. Lemarie-Rieusset 2002}, p. 227).
\begin{theorem}\label{th2} Let $E$ be a Banach space, and $B: E \times E \rightarrow  E$ 
be a continuous bilinear map such that there exists $\eta > 0$ so that
$$
\|B(x, y)\| \leq \eta \|x\| \|y\|,
$$
for all x and y in $E$. Then for any fixed $y \in E$ 
such that $\|y\| \leq \frac{1}{4\eta}$, the equation $x = y - B(x,x)$ 
has a unique solution  $\overline{x} \in E$ satisfying 
$\|\overline{x}\| \leq \frac{1}{2\eta}$.
\end{theorem}
\section{Proof of Theorem \ref{th1}}
In this section we shall give the proof of Theorem \ref{th1}. \\
We now need four more lemmas. In order to proceed, we define an auxiliary space  $\mathcal{N}^s_{p,T}$ which is made up 
of the functions $u(t,x)$ such that 
$$
u \in C([0, T);\dot{H}^s_p), \big\|u\big\|_{\mathcal{N}^s_{p,T}}:= \underset{0 < t < T}{{\rm sup}}\big\|u(t,.)\big\|_{\dot{H}^s_p} < \infty,
$$
and
\begin{equation}\label{eq5}
\underset{t \rightarrow 0}{\rm lim\ }\big\|u(t)\big\|_{\dot{H}^s_p} = 0,
\end{equation}
with $p > 1$ and $s \geq  \frac{d}{p} - 1$.\\
We define the auxiliary space $\mathcal{K}^{ \tilde q}_{q,T}$ 
which is made up of the functions $u(t,x)$ such that 
$$
t^{\frac{\alpha}{2}}u \in C([0, T); L^{\tilde q}),  \big\|u\big\|_{\mathcal{K}^{ \tilde q}_{q,T}}:= \underset{0 < t < T}{{\rm sup}}t^{\frac{\alpha}{2}}\big\|u(t,.)\big\|_{L^{\tilde q}} < \infty,
$$
and
\begin{equation}\label{eq6} 
\underset{t \rightarrow 0}{\rm lim\ }t^{\frac{\alpha}{2}}
\big\|u(t)\big\|_{L^{\tilde q}} = 0,
\end{equation}
with $\tilde q \geq q \geq d$ and  $\alpha = d(\frac{1}{q} - \frac{1}{\tilde q})$. 
\begin{remark}
The auxiliary space $\mathcal{K}_{\tilde q}: 
= \mathcal{K}^{\tilde q}_{d,T}\  (\tilde q \geq d)$ was introduced by 
Weissler and systematically used by Kato \cite{T. Kato 1984} 
and Cannone \cite{M. Cannone 1997}.
\end{remark}
\begin{lemma}\label{lem5} Suppose that $u_0 \in \dot{H}^s_p(\mathbb{R}^d)$
 with $p >1$ and $\frac{d}{p} - 1 \leq s < \frac{d}{p}$.
Then for all $\tilde q $ satisfying
$$
\tilde q > {\rm max}\{p,q\},
$$
where
$$
\frac{1}{q} = \frac{1}{p} - \frac{s}{d},
$$
we have
$$
\me^{t\Delta}u_0 \in \mathcal{K}^{\tilde q }_{q,\infty}.
$$
\end{lemma}
\begin{proof} First, we consider the case $p \leq q$. 
In this case $s \geq 0$, applying Lemma \ref{lem2}, we have $u_0 \in L^q$. We will prove that
$$
\underset{0 < t < \infty}{{\rm sup}}t^{\frac{\alpha}{2}}
\big\|\me^{t\Delta}u_0\big\|_{L^{\tilde q}} \lesssim 
\big\|u_0\big\|_{L^q},\ {\rm for\ all}\ \tilde q \geq q.
$$
Set
$$
\frac{1}{h} = 1 +  \frac{1}{\tilde q } - \frac{1}{q}.
$$
Applying Young's inequality we obtain
\begin{gather}
\big\|\me^{t\Delta}u_0\big\|_{L^{\tilde q}} = 
\frac{1}{(4\pi t)^{d/2}} \big\|\me^{\frac{-| . |^2}{4t}}*
u_0\big\|_{L^{\tilde q}} \notag \lesssim 
\frac{1}{t^{d/2}}\big\|\me^{\frac{-| . |^2}{4t}}
\big\|_{L^h}\big\|u_0\big\|_{L^q}  \notag \\
=t^{-\frac{\alpha}{2}}\big\|\me^{\frac{-| . |^2}{4}}\big\|_{L^h}
\big\|u_0\big\|_{L^{q}} \simeq t^{-\frac{\alpha}{2}}
\big\|u_0\big\|_{L^q}. \label{eq7}  
\end{gather}
This proves the result. We now prove that
$$
\underset{t \rightarrow 0}{\rm lim\ }t^{\frac{\alpha}{2}}
\big\|\me^{t\Delta}u_0\big\|_{L^{\tilde q}} = 0, \text{ for all}\ \tilde q > q.
$$
Set $\mathcal{X}_n(x) = 0$ for $x \in \{x : \ |x| < n\} \cap \{x : \ |u_0(x)| < n\}$ 
and $\mathcal{X}_n(x) = 1$ otherwise. We have
$$
t^{\frac{\alpha}{2}}\big\|\me^{t\Delta}u_0\big\|_{L^{\tilde q}} \leq 
C\big(t^{\frac{\alpha - d}{2}}\big\|\me^{\frac{-| . |^2}{4t}}*
(\mathcal{X}_nu_0)\big\|_{L^{\tilde q}} + 
t^{\frac{\alpha - d}{2}}\big\|\me^{\frac{-| . |^2}{4t}}*
((1 - \mathcal{X}_n)u_0)\big\|_{L^{\tilde q}}\big).
$$
Applying Young's inequality, we have
\begin{gather}
Ct^{\frac{\alpha - d}{2}}\big\|\me^{\frac{-| . |^2}{4t}}*(\mathcal{X}_nu_0)
\big\|_{L^{\tilde q }} \leq
C_1\big\|\me^{\frac{-| . |^2}{4}}\big\|_{L^h}\big\|\mathcal{X}_nu_0
\big\|_{L^q} \leq C_2\big\|\mathcal{X}_nu_0\big\|_{L^q}. \label{eq8}
\end{gather}
For any $\epsilon > 0$, we can take $n$ large enough that 
$C_2\big\|\mathcal{X}_nu_0\big\|_{L^q} < \frac{\epsilon}{2}$.\\
Fixed one of such $n$ and applying Young's inequality, we have
\begin{gather}
Ct^{\frac{\alpha - d}{2}}\big\|\me^{\frac{-| . |^2}{4t}}*
((1 - \mathcal{X}_n)u_0)
\big\|_{L^{\tilde q}} \leq
C_3t^{\frac{\alpha - d}{2}}\big\|\me^{\frac{-| . |^2}{4t}}\big\|_{L^1}
\big\|(1 - \mathcal{X}_n)u_0\big\|_{L^{\tilde q}}  \notag \\
\leq C_3t^{\frac{\alpha}{2}}\big\|\me^{\frac{-| . |^2}{4}}\big\|_{L^1}
\big\|n(1 - \mathcal{X}_n)\big\|_{L^{\tilde q}} = C_4(n)t^{\frac{\alpha}{2}} < \frac{\epsilon}{2},\ {\rm for}\ t<t_0  \label{eq9}
\end{gather}
with small enough $t_0=t_0(n)$.  From estimates \eqref{eq8} 
and \eqref{eq9}, we have
$$
t^{\frac{\alpha}{2}}\big\|\me^{t\Delta}u_0\big\|_{L^{\tilde q }} 
\leq C_2\big\|\mathcal{X}_nu_0\big\|_{L^{q}} + 
C_5(n)t^{\frac{\alpha}{2}} < \epsilon,\ {\rm for}\ t<t_0.
$$
Finally, we consider the case $p > q$. In this case $s < 0$, 
we will prove that
$$
\underset{0 < t < \infty}{{\rm sup}}t^{\frac{\alpha}{2}}
\big\|\me^{t\Delta}u_0\big\|_{L^{\tilde q}} \lesssim 
\big\|u_0\big\|_{\dot{H}^s_p},\ {\rm for}\ \tilde q \geq p.
$$
We have 
\begin{equation*}
\me^{t\Delta}u_0 =\me^{t\Delta}\dot{\Lambda}^{-s}\dot{\Lambda}^su_0 
= \frac{1}{t^{\frac{d - s}{2}}} K\Big(\frac{.}{\sqrt t}\Big) * 
(\dot{\Lambda}^su_0),
\end{equation*}
where
$$
\hat K(\xi) = \frac{1}{(2\pi)^{\frac{d}{2}}} \me^{-|\xi|^2} |\xi|^{- s},\ |K(x)| \lesssim \frac{1}{(1 + |x|)^{d - s}}.
$$
Set
$$
\frac{1}{h} = 1 +  \frac{1}{\tilde q} - \frac{1}{p}.
$$
Applying Young's inequality to obtain
\begin{gather*}
\big\|\me^{t\Delta}u_0\big\|_{L^{\tilde q }} \lesssim 
t^{-\frac{\alpha}{2}}\big\|K\big\|_{L^h}\big\|\dot{\Lambda}^su_0
\big\|_{L^p} \simeq t^{-\frac{\alpha}{2}}\big\|u_0\big\|_{\dot{H}^s_p}. 
\end{gather*}
This proves the result. We now claim that
\begin{equation*}
\underset{t \rightarrow 0}{\rm lim\ }t^{\frac{\alpha}{2}}
\big\|\me^{t\Delta}u_0\big\|_{L^{\tilde q }} = 0,\ \text{ for all}\ \tilde q > p.
\end{equation*}
Set  $\mathcal{X}_{n,s}(x) = 0$ for $x \in \{x : \ |x| < n\} 
\cap \{x : \ |\dot{\Lambda}^su_0(x)| < n\}$ and $\mathcal{X}_{n,s}(x) = 1$ 
otherwise. From the above proof we deduce that, for any $\epsilon > 0$, there exist a sufficiently large  $n$ 
and a sufficiently small $t_0 = t_0(n)$ such that
\begin{gather*}
t^{\frac{\alpha}{2}}\big\|\me^{t\Delta}u_0\big\|_{L^{\tilde q}} \leq  \notag \\
C_1\big\|K\big\|_{L^h}\big\|\mathcal{X}_{n,s}\dot{\Lambda}^s
u_0\big\|_{L^p} +  C_2nt^{\frac{d}{2}(\frac{1}{p} - \frac{1}{\tilde q})}
\big\|K\big\|_{L^1}\big\|1 - \mathcal{X}_{n,s}\big\|_{L^{\tilde q}} < \epsilon,\ {\rm for}\ t<t_0.
\end{gather*}
\end{proof}
In the following lemmas a particular attention will be devoted to the study 
of the bilinear operator $B(u, v)(t)$ defined by
\begin{equation}\label{eq10}
B(u, v)(t) = \int_{0}^{t} e^{(t-\tau ) \Delta} \mathbb{P} 
\nabla\cdot\big(u(\tau)\otimes v(\tau)\big) \dif\tau.
\end{equation}
\begin{lemma}\label{lem6}
Let $p$ and $s$ be such that
\begin{equation*}
p > \frac{d}{2}\ \text{and}\ \frac{d}{p} - 1 \leq s  < \frac{d}{2p}.
\end{equation*}
Then the bilinear operator $B$ is continuous 
from $\mathcal{K}^{\tilde q}_{q,T} \times 
\mathcal{K}^{\tilde q}_{q,T}$ into 
$\mathcal{N}^s_{p,T}$, where 
\begin{equation*}
\frac{1}{q} = \frac{1}{p} - \frac{s}{d},\ q < \tilde q  \leq 2p,
\end{equation*}
and the following inequality holds
\begin{equation}\label{eq11}
\big\|B(u, v)\big\|_{\mathcal{N}^s_{p,T}} \leq
 CT^{\frac{1}{2}(1 + s  - \frac{d}{p})}
\big\|u\big\|_{\mathcal{K}^{\tilde q}_{q,T}}
\big\|v\big\|_{\mathcal{K}^{\tilde q}_{q,T}},
\end{equation}
where C is a positive constant and independent of T.
\end{lemma}
\begin{proof} By Lemma \ref{lem1}, we have 
\begin{gather}
\big\|B(u,v)(t)\big\|_{\dot{H}^s_p} \leq \int_{0}^{t} 
\big\|\dot\Lambda^s\me^{(t- \tau) \Delta} \mathbb{P} \nabla\cdot\big(u(\tau)\otimes 
v(\tau)\big)\big\|_{L^p}{\rm d}\tau   \notag \\
 =\int_{0}^{t} \Big\|\frac{1}{(t - \tau)^{\frac{d + 1 + s}{2}}}K\Big(\frac{.}
{\sqrt{t - \tau}}\Big)*\big(u(\tau) \otimes v(\tau)\big)\Big\|_{L^p}{\rm d}\tau. \label{eq12}
\end{gather}
Applying Young's inequality, we have
\begin{gather}
\Big\|\frac{1}{(t - \tau)^{\frac{d + 1 + s}{2}}}K\Big(\frac{.}
{\sqrt{t - \tau}}\Big)*\big(u(\tau) \otimes v(\tau)\big)\Big\|_{L^p}  \notag \\
\lesssim\frac{1}{(t - \tau)^{\frac{d + 1 + s}{2}}}
\Big\|K\Big(\frac{.}{\sqrt{t - \tau}}\Big)\Big\|_{L^r}
\big\|u(\tau) \otimes v(\tau)\big\|_{L^{{\frac{\tilde q}{2}}}},\label{eq13}
\end{gather}
where
\begin{gather}
\frac{1}{r} = 1+ \frac{1}{p} - \frac{2}{\tilde q},\label{eq14}
\end{gather}
Applying H\"{o}lder's inequality, we have
\begin{gather}
\big\|u(\tau) \otimes v(\tau)\big\|_{L^{\frac{\tilde q}{2}}} 
\leq \big\|u(\tau)\big\|_{L^{\tilde q}}\big\|v(\tau)\big\|_{L^{\tilde q}}.
\label{eq15}
\end{gather}
Since the equality \eqref{eq14} and Lemma \ref{lem1}  it follows that
\begin{gather}
\Big\|K\Big(\frac{.}{\sqrt{t - \tau}}\Big)\Big\|_{L^r} 
=  (t - \tau)^{\frac{d}{2r}}\big\|K\big\|_{L^r} 
\simeq (t - \tau)^{\frac{d}{2}(1+ \frac{1}{p} - \frac{2}{\tilde q})}.
\label{eq16}
\end{gather}
From the inequalities \eqref{eq13}, \eqref{eq15}, 
and \eqref{eq16} we deduce that 
\begin{gather}
\big\|\me^{(t - \tau)\Delta}\mathbb{P}\nabla\cdot\big(u(\tau) 
\otimes v(\tau)\big)\big\|_{\dot{H}^s_p} 
\lesssim (t - \tau)^{\frac{d}{2p} - \frac{d}{\tilde q} 
- \frac{s + 1}{2}}\big\|u(\tau)\big\|_{L^{\tilde q}}
\big\|v(\tau)\big\|_{L^{\tilde q}}.\label{eq17}
\end{gather}
By the inequalities \eqref{eq12}, \eqref{eq17},  and Lemma \ref{lem4}, we have
\begin{gather}
\big\|B(u, v)(t)\big\|_{\dot{H}^s_p} \lesssim  
\int_0^t (t - \tau)^{\frac{d}{2p} - \frac{d}{\tilde q} 
- \frac{s + 1}{2}}\big\|u(\tau)\big\|_{L^{\tilde q}}
\big\|v(\tau)\big\|_{L^{\tilde q}}\dif\tau \notag \\
\leq \int_0^t (t - \tau)^{\frac{d}{2p} - \frac{d}{\tilde q} - 
\frac{s + 1}{2}}\tau^{-\alpha}\underset{0 < \eta < t}{{\rm sup}}
\eta^{\frac{\alpha}{2}}\big\|u(\eta)\big\|_{L^{\tilde q}}
\underset{0 < \eta < t}{{\rm sup}}\eta^{\frac{\alpha}{2}}
\big\|v(\eta)\big\|_{L^{\tilde q}}\dif\tau \notag \\
= \underset{0 < \eta < t}{{\rm sup}}\eta^{\frac{\alpha}{2}}
\big\|u(\eta)\big\|_{L^{\tilde q}}\underset{0 < \eta < t}{{\rm sup}}
\eta^{\frac{\alpha}{2}}\big\|v(\eta)\big\|_{L^{\tilde q}}
\int_0^t (t - \tau)^{\frac{d}{2p} - \frac{d}{\tilde q} 
- \frac{s + 1}{2}}\tau^{-\alpha}\dif\tau \notag \\
\simeq  t^{\frac{1}{2}(1 + s - \frac{d}{p})}
\underset{0 < \eta < t}{{\rm sup}}\eta^{\frac{\alpha}{2}}
\big\|u(\eta)\big\|_{L^{\tilde q}}\underset{0 < \eta < t}{{\rm sup}}
\eta^{\frac{\alpha}{2}}\big\|v(\eta)\big\|_{L^{\tilde q}}.  \label{eq18}
\end{gather}
The estimate \eqref{eq11} is deduced from the inequality \eqref{eq18}.\\ 
Let us now check the validity of condition \eqref{eq5} 
for the bilinear term $B(u,v)(t)$.\\
In fact, from the estimate \eqref{eq18} it follows that
\begin{equation}\label{eq19}
\underset{t \rightarrow 0}{\rm lim\ }
\big\|B(u,v)(t)\big\|_{\dot{H}^s_p} = 0,
\end{equation}
whenever
$$
\underset{t \rightarrow 0}{\rm lim\ }
t^{\frac{\alpha}{2}}\big\|u(t)\big\|_{L^{\tilde q}} 
= \underset{t \rightarrow 0}{\rm lim\ }t^{\frac{\alpha}{2}}
\big\|v(t)\big\|_{L^{\tilde q}} = 0.
$$
Finally, the continuity at $t=0$ of $B(u, v)(t)$ follows from the equality \eqref{eq19}. The continuity elsewhere follows from carefully rewriting the expression 
$\int^{t+\epsilon}_0 - \int^{t}_0$ and applying the same argument. \end{proof} 
\begin{lemma}\label{lem7}
Let $q$ and $\tilde q$ be such that $\tilde q > q \geq d$. Then the bilinear operator $B$ is continuous from 
$\mathcal{K}^{\tilde q}_{q,T} \times \mathcal{K}^{\tilde q}_{q,T}$ 
into $\mathcal{K}^{\tilde q}_{q,T}$ and the following inequality holds
\begin{equation}\label{eq20}
\big\|B(u, v)\big\|_{\mathcal{K}^{\tilde q}_{q,T}} \leq
 CT^{\frac{1}{2}(1  - \frac{d}{q})}
\big\|u\big\|_{\mathcal{K}^{\tilde q}_{q,T}}
\big\|v\big\|_{\mathcal{K}^{\tilde q}_{q,T}},
\end{equation}
where C is a positive constant and independent of T.
\end{lemma}
\begin{proof} Applying the estimate \eqref{eq17} 
for $s = 0\ {\rm and}\ p = \tilde q$, we have
\begin{gather*}
\big\|\me^{(t - \tau)\Delta}\mathbb{P}\nabla\cdot\big(u(\tau)  
\otimes v(\tau)\big)\big\|_{L ^{\tilde q}} 
\lesssim (t - \tau)^{- \frac{d}{2\tilde q} - \frac{1}{2}}
\big\|u(\tau)\big\|_{L^{\tilde q}}\big\|v(\tau)\big\|_{L^{\tilde q}}.
\end{gather*}
Applying Lemma \ref{lem4}, we have
\begin{gather*}
\big\|B(u, v)(t)\big\|_{L^{\tilde q}} \lesssim  
\int_0^t (t - \tau)^{- \frac{d}{2\tilde q} - \frac{1}{2}}
\big\|u(\tau)\big\|_{L^{\tilde q}}.\big\|v(\tau)\big\|_{L^{\tilde q}}
\dif\tau \notag \\ 
\leq \int_0^t (t - \tau)^{- \frac{d}{2\tilde q} - \frac{1}{2}}
\tau^{-\alpha}\underset{0 < \eta< t}{{\rm sup}}\eta^{\frac{\alpha}{2}}
\big\|u(\eta)\big\|_{L^{\tilde q}}\underset{0 < \eta < t}{{\rm sup}}
\eta^{\frac{\alpha}{2}}\big\|v(\eta)\big\|_{L^{\tilde q}}\dif\tau \notag \\
= \underset{0 < \eta < t}{{\rm sup}}\eta^{\frac{\alpha}{2}}
\big\|u(\eta)\big\|_{L^{\tilde q}}\underset{0 < \eta < t}{{\rm sup}}
\eta^{\frac{\alpha}{2}}\big\|v(\eta)\big\|_{L^{\tilde q}}
\int_0^t (t - \tau)^{- \frac{d}{2\tilde q} - \frac{1}{2}}
\tau^{-\alpha}\dif\tau \notag \\
 \simeq  t^{-\frac{\alpha}{2}}t^{\frac{1}{2}(1  - \frac{d}{q})}
\underset{0 < \eta < t}{{\rm sup}}\eta^{\frac{\alpha}{2}}
\big\|u(\eta)\big\|_{L^{\tilde q}}\underset{0 < \eta < t}{{\rm sup}}
\eta^{\frac{\alpha}{2}}\big\|v(\eta)\big\|_{L^{\tilde q}}.
\end{gather*}
Thus
\begin{gather}
t^{\frac{\alpha}{2}}\big\|B(u, v)(t)\big\|_{L^{\tilde q}}
 \lesssim  t^{\frac{1}{2}(1  - \frac{d}{q})}
\underset{0 < \eta < t}{{\rm sup}}\eta^{\frac{\alpha}{2}}
\big\|u(\eta)\big\|_{L^{\tilde q}}\underset{0 < \eta < t}{{\rm sup}}
\eta^{\frac{\alpha}{2}}\big\|v(\eta)\big\|_{L^{\tilde q}}. \label{eq21}
\end{gather}
The estimate \eqref{eq20} is deduced from the inequality \eqref{eq21}.\\
Now we check the validity of condition \eqref{eq6} for the bilinear term $B(u,v)(t)$. From the estimate \eqref{eq21} it follows that 
$$
\underset{t \rightarrow 0}{\rm lim\ }t^{\frac{\alpha}{2}}
\big\|B(u, v)(t)\big\|_{L^{\tilde q}} = 0,
$$
whenever
$$
\underset{t \rightarrow 0}{\rm lim\ }t^{\frac{\alpha}{2}}
\big\|u(t)\big\|_{L^{\tilde q}} = \underset{t \rightarrow 0}{\rm lim\ }
t^{\frac{\alpha}{2}}\big\|v(t)\big\|_{L^{\tilde q}} = 0.
$$
Finally, the continuity at $t=0$ of $t^{\frac{\alpha}{2}}B(u, v)(t)$ follows from the equality \eqref{eq19}. The continuity elsewhere follows from carefully rewriting the expression $\int^{t+\epsilon}_0 - \int^{t}_0$ and applying the same argument.
\end{proof} 
The following lemma, the proof of which is omitted, is a  
generalization of Lemma \ref{lem7}.
\begin{lemma}\label{lem8} Let $d \leq q \leq \tilde q_2 < \infty$ 
and $q <  \tilde q_1 < \infty$ be such that one of the following conditions is satisfied. 
$$
q < \tilde q_1 < 2d, q \leq \tilde q_2 < \frac{d\tilde q_1}{2d - \tilde q_1}, 
$$
or
$$
2d \leq \tilde q_1 \leq 2q, q \leq \tilde q_2 < \infty,
$$
or
$$
2q < \tilde q_1 < \infty, \frac{\tilde q_1}{2} < \tilde q_2 < \infty.
$$
Then the bilinear operator $B$ is continuous from 
$\mathcal{K}^{\tilde q_1}_{q,T} \times 
\mathcal{K}^{\tilde q_1}_{q,T}$ into 
$\mathcal{K}^{\tilde q_2}_{q,T}$, and we have the inequality
\begin{equation*}
\big\|B(u, v)\big\|_{\mathcal{K}^{\tilde q_2}_{q,T}} \leq
 CT^{\frac{1}{2}(1  - \frac{d}{q})}
\big\|u\big\|_{\mathcal{K}^{\tilde q_1}_{q,T}}
\big\|v\big\|_{\mathcal{K}^{\tilde q_1}_{q,T}},
\end{equation*}
where C is a positive constant and independent of T.
\end{lemma}
\vskip 0.5cm
{\bf Proof of Theorem \ref{th1}}
\vskip 0.5cm
(a) From Lemma \ref{lem7}, $B$ is 
continuous from $\mathcal{K}^{\tilde q}_{q,T} 
\times \mathcal{K}^{\tilde q}_{q,T}$ to 
$\mathcal{K}^{\tilde q}_{q,T}$ and we have the inequality
$$
\big\|B(u, v)\big\|_{\mathcal{K}^{\tilde q}_{q,T}} 
\leq C_{q,\tilde q, d}T^{\frac{1}{2}(1-\frac{d}{q})}
\big\|u\big\|_{\mathcal{K}^{\tilde q}_{q,T}}
\big\|v\big\|_{\mathcal{K}^{\tilde q}_{q,T}} =
 C_{q,\tilde q, d}T^{\frac{1}{2}(1+s-\frac{d}{p})}
\big\|u\big\|_{\mathcal{K}^{\tilde q}_{q,T}}
\big\|v\big\|_{\mathcal{K}^{\tilde q}_{q,T}},
$$
where $C_{q,\tilde q,d}$ is a positive constant independent 
of  $T$. From Theorem \ref{th2} and the above inequality, 
we deduce that for any $u_0 \in  \dot{H}^s_p$ satisfying
$$
T^{\frac{1}{2}(1+s-\frac{d}{p})}
\big\|\me^{t\Delta}u_0\big\|_{\mathcal{K}^{\tilde q}_{q,T}} 
= T^{\frac{1}{2}(1+s-\frac{d}{p})}\underset{0 < t < T}{{\rm sup}}
t^{\frac{\alpha}{2}}\big\|\me^{t\Delta}u_0\big\|_{L^{\tilde q}} 
\leq  \frac{1}{4C_{q,\tilde q, d}},
$$
where
$$
\alpha = d\Big(\frac{1}{q}- \frac{1}{\tilde q}\Big) 
= d\Big(\frac{1}{p}- \frac{s}{d} - \frac{1}{\tilde q}\Big),
$$
NSE has a solution $u$ on the interval $(0, T)$ 
so that $u \in \mathcal{K}^{\tilde q}_{q,T}$.\\
We prove that  $u \in \underset{r > {\rm max}\{p, q\}}
{\bigcap}\mathcal{K}^r_{q,T}$. We consider three cases 
$q < \tilde q < 2d$ and $2d \leq \tilde q \leq 2q$, 
and $2q < \tilde q < \infty$ separately.\\
Note that if ${\rm max}\{p,q\} \geq 2d$ then there does not exist 
$\tilde q$ satisfying the condition of the first case, 
and if  $p \geq 2q$ then there does not exist $\tilde q$ satisfying 
the condition of the second case. Therefore the number of cases 
that can occur depends on $s$ and $p$.\\
First, we consider the case $q < \tilde q < 2d$. 
There are two possibilities $\tilde q > \frac{4d}{3}$ 
and $\tilde q \leq \frac{4d}{3}$.  In the case $\tilde q > \frac{4d}{3}$, 
we apply  Lemmas \ref{lem5} and \ref{lem8} 
to obtain $u \in \mathcal{K}^r_{q,T}$ for all $r$ satisfying 
${\rm max}\{p, q\} < r < \tilde q_1 $ where 
$\tilde q_1 = \frac{d\tilde q}{2d - \tilde q} > 2d$. 
Thus, $u \in \mathcal{K}^{2d}_{q,T}$. Applying again 
Lemmas \ref{lem5} and \ref{lem8}, we deduce that
 $u \in \mathcal{K}^r_{q,T}$ for all $r > {\rm max}\{p, q\}$. In the case 
$\tilde q \leq \frac{4d}{3}$, we set up the following series of numbers  
$\{\tilde q_i\}_{0 \leq i \leq N}$ by induction. Set $\tilde q_0 = \tilde q$ 
and $\tilde q_1 = \frac{d\tilde q_0}{2d - \tilde q_0}$. 
We have $\tilde q_1 >  \tilde q_0$. If $\tilde q_1 > \frac{4d}{3}$ 
then set $N=1$ and stop here. In the case $\tilde q_1 \leq \frac{4d}{3}$ 
set $\tilde q_2 = \frac{d\tilde q_1}{2d - \tilde q_1}$. 
We have $\tilde q_2 >  \tilde q_1$. If $\tilde q_2 > \frac{4d}{3}$ then set 
$N=2$ and stop here. In the case $\tilde q_2 \leq \frac{4d}{3}$,   
set $\tilde q_3 = \frac{d\tilde q_2}{2d - \tilde q_2}$.  
We have $\tilde q_3 >  \tilde q_2$, and so on, there exists  $k \geq 0$ 
such that $\tilde q_k \leq \frac{4d}{3}, \tilde q_{k+1}   
= \frac{d\tilde q_k}{2d - \tilde q_k} > \frac{4d}{3}$. 
We set $N = k+1$ and stop here, and we have
\begin{gather*}
\tilde q_0 = \tilde q, \tilde q_i = \frac{d\tilde q_{i-1}}{2d 
- \tilde q_{i-1}}, \tilde q_i > \tilde q_{i-1}\ {\rm for}\ i =1,2,3,..,N,\\
2d \geq \tilde q_{N} > \frac{4d}{3} \geq \tilde q_{N-1}.
\end{gather*}
From  $u \in \mathcal{K}^{\tilde q_0}_{q,T}$,  applying 
Lemmas \ref{lem5} and \ref{lem8} to obtain 
$u \in \mathcal{K}^r_{q,T}$ for all $r$ satisfying 
${\rm max}\{p, q\} < r < \tilde q_1$.  Then applying again 
Lemmas \ref{lem5} and \ref{lem8} to get 
$u \in \mathcal{K}^r_{q,T}$ for all $r$ satisfying 
${\rm max}\{p, q\} < r < \tilde q_2$, and so on, 
finishing we have $u \in \mathcal{K}^r_{q,T}$ for all $r$ 
satisfying ${\rm max}\{p, q\} < r < \tilde q_N$. 
Therefore $u \in \mathcal{K}^r_{q,T}$ for all $r$ satisfying 
$\frac{4d}{3} < r < \tilde q_N$. From the proof of the 
case $\tilde q > \frac{4d}{3}$, we have $u \in \mathcal{K}^r_{q,T}$ 
for all $r > {\rm max}\{p, q\}$.\\
We now consider the case $2d \leq \tilde q \leq 2q$. 
We show that $u \in \mathcal{K}^r_{q,T}$ for all $ r > {\rm max}\{p, q\}$. 
This is easily deduced by applying Lemmas \ref{lem5} and \ref{lem8}.\\
Finally, we consider the case $2q < \tilde q < \infty$. 
Let $i \in \mathbb{N}$ be such that 
$$
\frac{\tilde q}{2^{i-1}} \geq {\rm max}\{2q,p\} > \frac{\tilde q}{2^i}.
$$
From $\tilde q > {\rm max}\{p, q\}$ and $\tilde q>2q$, we have 
$\tilde q>{\rm max}\{2q,p\}$, hence $i \geq 1$. 
Applying Lemmas \ref{lem5} and \ref{lem8} to obtain 
$u \in \mathcal{K}^r_{q,T}$ for all $ r > \frac{\tilde q}{2}$. 
Applying again Lemmas \ref{lem5} and \ref{lem8} to get 
$u \in \mathcal{K}^r_{q,T}$ for all $ r > \frac{\tilde q}{2^2}$, 
and so on, finishing we have $u \in \mathcal{K}^r_{q,T}$ 
for all $r > \frac{\tilde q}{2^{i-1}}$. Applying again 
Lemmas \ref{lem5} and \ref{lem8} to 
obtain $u \in \mathcal{K}^r_{q,T}$ for all 
$r > {\rm max}\{p,q, \frac{\tilde q}{2^i}\}$. 
If ${\rm max}\{p,q\} \geq \frac{\tilde q}{2^i}$ 
then we have $u \in \mathcal{K}^r_{q,T}$ for all 
$r > {\rm max}\{p,q\}$. If ${\rm max}\{p,q\} < \frac{\tilde q}{2^i}$ 
then $2q > \frac{\tilde q}{2^i}$. Thus $u \in \mathcal{K}^r_{q,T}$ 
for all $r$ satisfying $r>\frac{\tilde q}{2^i}$, 
hence $u \in \mathcal{K}^{2q}_{q,T}$. Applying 
Lemmas \ref{lem5} and \ref{lem8} to obtain 
$u \in \mathcal{K}^r_{q,T}$ for all $r > {\rm max}\{p,q\}$. This proves the result.\\
We now prove that $u \in BC([0, T); \dot{H}^s_p)$. 
Indeed, from $u \in \mathcal{K}^r_{q,T}$ for all $r > {\rm max}\{p,q\}$, applying Lemma \ref{lem6} to obtain
$B(u,u) \in \mathcal{N}^s_{p,T} \subseteq BC\big([0, T); 
\dot{H}^s_p\big)$. On the other hand, since $u \in \dot{H}^s_p$, it follows that  $\me^{t\Delta}u_0 \in BC\big([0, T); \dot{H}^s_p\big)$. Therefore
$$
u = \me^{t\Delta}u_0 - B(u,u) \in  BC\big([0, T); \dot{H}^s_p\big).
$$
Finally, we will show that the condition \eqref{eq2} is valid when $T$ is small enough. Indeed, from the definition of $\mathcal{K}^{ \tilde q}_{q,T}$ 
and Lemma \ref{lem5}, 
we deduce that the left-hand side of the condition \eqref{eq2} 
converges to $0$ when $T$ goes to $0$. Therefore the 
condition \eqref{eq2} holds for 
arbitrary $u_0 \in \dot{H}^s_p (\mathbb{R}^d)$ 
when $T(u_0)$ is small enough. \\
(b) From Lemma \ref{lem3}, the two quantities 
$\big\|u_0\big\|_{\dot{B}^{\frac{d}{\tilde q} - 1, \infty}_{\tilde q}}$ 
and  $\underset{0 < t < \infty}{{\rm sup}}t^{\frac{1}{2}(1- \frac{d}{\tilde q})}\big\|\me^{t\Delta}u_0\big\|_{L^{\tilde q}}$  are equivalent. 
Thus, there exists a positive constant $\sigma_{\tilde q,d}$ 
such that the condition
\eqref{eq2} holds for $T = \infty$ whenever $\big\|u_0\big\|_{\dot{B}^{\frac{d}
{\tilde q} - 1, \infty}_{\tilde q}} \leq \sigma_{\tilde q,d}$.\qed 
\vskip 0.2cm
{\bf Proof of Proposition \ref{pro1}}
\vskip 0.2cm
By Theorem \ref{th1}, we only need to prove that $w \in \mathcal{N}^{\frac{d}{\tilde p}-1}_{{\tilde p},T}$ 
for all $\tilde p >  \frac{1}{2}{\rm max}\{p, d\}$. Indeed, applying Lemma \ref{lem6}, we deduce that the bilinear operator $B$ is continuous from $\mathcal{K}^r_{d,T} \times \mathcal{K}^r_{d,T}$ into $\mathcal{N}^{\frac{d}
{\tilde p}-1}_{{\tilde p},T}$ for all $\tilde p > \frac{d}{2}$ and $r$ satisfying  $d < r \leq 2\tilde p$, hence from $u \in \underset{r > {\rm max}\{p,d\}}{\bigcap}\mathcal{K}^r_{d,T}$ and $2\tilde p > {\rm max}\{p,d\}$, we have $w=-B(u,u) 
\in \mathcal{N}^{\frac{d}{\tilde p}-1}_{{\tilde p},T}$. The proof Proposition \ref{pro1} is complete. \qed  
\vskip 0.5cm
{\bf Proof of Proposition \ref{pro2}}
\vskip 0.5cm
By Lemma \ref{lem3}, we deduce 
that the two quantities $\big\|u_0\big\|_{\dot{B}^{s- 
(\frac{d}{p} - \frac{d}{\tilde q}), \infty}_{\tilde q}}$\linebreak 
and $\underset{0 < t < \infty}{{\rm sup}}
t^{\frac{d}{2}(\frac{1}{p}- \frac{s}{d} 
- \frac{1}{\tilde q})}\big\|\me^{t\Delta}u_0\big\|_{L^{\tilde q}}$ 
are equivalent. Therefore
$$
\underset{0 < t < T}{{\rm sup}}t^{\frac{d}{2}
(\frac{1}{p}- \frac{s}{d} - \frac{1}{\tilde q})}
\big\|\me^{t\Delta}u_0\big\|_{L^{\tilde q}} 
\lesssim \big\|u_0\big\|_{\dot{B}^{s- (\frac{d}{p} 
- \frac{d}{\tilde q}), \infty}_{\tilde q}}.
$$
Proposition \ref{pro2} is proved by applying the above 
inequality and Theorem \ref{th1}. \qed

\section*{Acknowledgments.} This research was supported by Vietnam
National Foundation for Science and Technology Development
(NAFOSTED) under grant number 101.02-2014.50.

\end{document}